\documentclass[12pt,a4paper,reqno]{amsart}%
\usepackage{amsfonts}
\usepackage{amsmath}
\usepackage{amssymb}
\usepackage{graphicx}%
\newtheorem{theorem}{Theorem}[section]
\newtheorem{lemma}[theorem]{Lemma}

\theoremstyle{definition}

\newtheorem{corollary}[theorem]{Corollary}
\newtheorem{proposition}[theorem]{Propostion}

\theoremstyle{remark}
\newtheorem{remark}[theorem]{Remark}

\newenvironment{Acknowledgement}[1][Acknowledgement]{\noindent\textit{#1.} }

\numberwithin{equation}{section}

\begin{document}
\title [Prescribing the symmetric function of the Schouten tensor]{Prescribing the symmetric function of the eigenvalues of the Schouten tensor}

\author{Yan He}
\address{Centre for Mathematical Sciences, Zhejiang
University, Hangzhou 310027, P. R. China.}

\email{helenaig@zju.edu.cn}

\author{Weimin Sheng }
\address{Department of Mathematics, Zhejiang University, Hangzhou 310027, P. R. China
}
\email{weimins@zju.edu.cn}
\thanks{This work was partially supported by NSFC Grant \# 10771189 and 10831008.}


\date{February 1, 2010}


\begin{abstract}
In this paper we study the problem of conformally deforming a metric to a
prescribed symmetric function of the eigenvalues of the Schouten tensor on
compact Riemannian manifolds with boundary. We prove its solvability and the
compactness of the solution set, provided the Ricci tensor is non-negative
definite.
\end{abstract}

\maketitle


\bibliographystyle{amsplain}

\section{INTRODUCTION}

Let $(M^{n},g)$ be a smooth, compact Riemannian manifold with totally
geodesic boundary of dimension $n\geq 3$. The Schouten tensor of $g$ is
defined by
\begin{equation*}
A_{g}={\frac{1}{{n-2}}}\left( \mathrm{Ric}_{g}-{\frac{{R_{g}}}{{2(n-1)}}}%
g\right) ,
\end{equation*}%
where $\mathrm{Ric}$ and $R$ are the Ricci and scalar curvatures of $g$,
respectively.

Let $\sigma _{k}:R^{n}\rightarrow R$ be the $k$-th elementary symmetric
function ($1\leq k\leq n$)
\begin{equation*}
\sigma _{k}(x)=\sum_{1\leq i1<\cdots < ik\leq n} x_{i1}\cdots
x_{ik},
\end{equation*}%
$\Gamma_{k}$ the corresponding open, convex cone, i.e. $\Gamma _{k}=$ $%
\left\{ x\in R^{n}|\sigma _{i}\left( x\right) >0,1\leq i\leq k\right\} .$
Let
\begin{equation*}
\Sigma _{\theta }=\left\{ x=\left( x_{1},\cdots ,x_{n}\right) \in
R^{n}|\quad \min x_{i}+\theta \Sigma x_{i}>0\right\} .
\end{equation*}%
Now let us consider the general symmetric function $F$ defining on $\Gamma $
$(\Gamma _{n}\subset \Gamma \subset \Sigma_{\frac{1}{n-2}})$ satisfying

($C_1$) $F$ is positive and $F=0$ on $\partial \Gamma$;

($C_2$) $F$ is concave;

($C_3$) $F$ is invariant under exchange of variables;

($C_4$) $F$ is homogeneous of degree 1;

($C_5$)  $\frac{\partial F}{\partial x_i}\geq \varepsilon\frac{F}{\sigma_1}$ for some constant $\varepsilon>0$ for all $i$;

($C_6$) $F(x)\leq \varrho \sigma_1(x)$ in $\Gamma$ and $F(1,\cdots,1)= n
\varrho$, $\varrho$ is a positive constant.

We need ($C_{1}$)-($C_{4}$) to ensure that the elliptic equations are
solvable. $F=\sigma_k^{1/k}$ satisfies condition ($C_{5}$). The condition ($C_{6}$) says that the
Newton-Maclaurin inequality with respect to function $F$ holds. 

We denote $[g]=\{\tilde{g}\mid \tilde{g}=e^{-2u}g\}$. We call the metric $%
\hat{g}=e^{-2u}g$ (as well as the function $u$) is $\Gamma $-admissible, or
simply admissible, if $\hat{g}\in \left\{ \tilde{g}\in \left[ g\right] \quad
|\quad \lambda (\tilde{g}^{-1}A_{\tilde{g}})\in \Gamma \right\} $. Here, $%
\lambda (\tilde{g}^{-1}A_{\tilde{g}})=\left( \lambda _{1},\cdots ,\lambda
_{n}\right) $ denote the eigenvalues of $\tilde{g}^{-1}A_{\tilde{g}}$.

In this paper we study the existence of some prescribing problems and
the compactness of the solution set. The main result is as follows.

\begin{theorem}
\textit{Let }$(M^{n},g)$\textit{\ be a compact }$n$%
\textit{-dimensional Riemannian manifold with totally geodesic boundary. Let
}$F$\textit{\ be a symmetric function satisfying }$(C_{1})-(C_{6})$\textit{%
\ on }$\Gamma $\textit{\ with }$\Gamma _{n}\subset \Gamma \subset \Sigma _{\frac
{1} {n-2}}$\textit{\ . If the manifold }$(M,g)$\textit{\ is not
conformal equivalent to a hemisphere, then for any positive function }$f$%
\textit{, there exists an admissible conformal metric }$\tilde{g}=e^{-2u}g$%
\textit{\ with totally geodesic boundary satisfying }%
\begin{equation*}
F\left( \lambda (\tilde{g}^{-1}A_{\tilde{g}})\right) =f.
\end{equation*}%
\textit{Additionally, the set of all such solutions is compact in the }$%
C^{m} $\textit{\ -topology for any }$m\geq 0$\textit{.\newline
}
\end{theorem}

We can get the following corollary from Theorem 1.1 immediately. That
is to find a conformal metric $\tilde{g}$ with nonnegative $Ric_{\tilde{g}}$
such that
\begin{equation}
\det \left( \mu (\tilde{g}^{-1}Ric_{\tilde{g}})\right) =f^{n},  \tag{1.1}
\label{1.1}
\end{equation}%
where $\mu (\tilde{g}^{-1}Ric_{\tilde{g}})=\left( \mu _{1},\cdots ,\mu
_{n}\right) $ are the eigenvalues of $\tilde{g}^{-1}Ric_{\tilde{g}}$ and $%
f(x)$ is a positive function.

Since $Ric_{\tilde{g}}=(n-2)A_{\tilde{g}}+\sigma _{1}(\lambda (\tilde{g}%
^{-1}A_{\tilde{g}}))\tilde{g}$, if we define \\
$F(\lambda )=\sigma
_{n}^{1/n}\left( (n-2)\lambda +\left( \Sigma _{i=1}^{n}\lambda _{i}\right)
\right) $ and $\Gamma =\{\lambda \quad |\quad F(\lambda )>0\}$, then
\begin{equation*}
\begin{array}{ll}
\displaystyle & \det^{1/n}\left( \mu (\tilde{g}^{-1}Ric_{\tilde{g}})\right)
\\
\displaystyle= & \sigma _{n}^{1/n}\Big(\mu \Big(g^{-1}\Big[(n-2)(du\otimes
du-|\nabla u|^{2}g)+(n-2)\nabla ^{2}u+\triangle ug+Ric_{g}\Big]\Big)\Big) \\
=\displaystyle & F\Big(\lambda \Big(g^{-1}\Big[\nabla ^{2}u+du\otimes du-{%
\frac{1}{2}}|\nabla u|^{2}g+A_{g}\Big]\Big)\Big),%
\end{array}%
\end{equation*}%
where $\mu =(n-2)\lambda +\Sigma _{i=1}^{n}\lambda _{i}$. From the
definition, it is easy to verify that $F$ satisfying $(C_{1})-(C_{5})$,
since
\begin{equation*}
\frac{\partial F}{\partial \lambda _{i}}=\frac{\partial \left( \sigma
_{n}^{1/n}\right) }{\partial \mu _{s}}\left( 1+(n-2)\delta _{i}^{s}\right) ,
\end{equation*}%
and
\begin{equation*}
\frac{\partial ^{2}F}{\partial \lambda _{i}\partial \lambda _{j}}=\left(
1+(n-2)\delta _{i}^{s}\right) \frac{\partial ^{2}\left( \sigma
_{n}^{1/n}\right) }{\partial \mu _{s}\partial \mu _{t}}\left( 1+(n-2)\delta
_{j}^{t}\right) .
\end{equation*}%
Moreover, from
\begin{eqnarray*}
F\left( \lambda (\tilde{g}^{-1}A_{\tilde{g}})\right)  &=&\displaystyle\sigma
_{n}^{1/n}\left( \mu (\tilde{g}^{-1}Ric_{\tilde{g}})\right)  \\
&\leq &\frac{1}{n}\sigma _{1}\left( \mu (\tilde{g}^{-1}Ric_{\tilde{g}%
})\right)  \\
&=&\frac{2n-2}{n}\sigma _{1}\left( \lambda (\tilde{g}^{-1}A_{\tilde{g}%
})\right) ,
\end{eqnarray*}%
we know $F$ satisfies $(C_{6})$ with $\varrho =  \frac{2n-2 } {n}$. Thus (\ref%
{1.1}) turns out to be a proper equation with respect to Schouten tensor.
Furthermore, as \cite{GV1} and \cite{TW1}, by use of the volume comparison
theorem, $C^{0}$ estimate of the solutions of such an equation can be derived
if $Ric\geq 0$. In other words, the condition $\Gamma \subset \Sigma _{\frac{
1} {n-2}}$ ensures that the volume comparison theorem is applicable, where
the eigenvalues of Schouten tensor $\lambda $ satisfy  $(n-2)\lambda +\Sigma
_{i=1}^{n}\lambda _{i}\geq 0$ if and only if the eigenvalues of Ricci tensor
$\mu \geq 0$. Similarly, on the manifold with totally geodesic boundary,
based on the boundary $C^{1},C^{2}$ estimates with Neumann boundary
condition for general symmetric function (\cite{Cn2} or \cite{HS1}, etc.),
we can get

\begin{corollary}
 \textit{Let }$(M,g)$\textit{\ be a compact }$n$%
\textit{\ dimension Riemannian manifold with totally geodesic boundary and
the Ricci tensor is semi-positive definite. If it is not conformal
equivalent to a hemisphere, then for any positive function }$f$\textit{,
there exists a conformal metric }$\tilde{g}=e^{-2u}g$\textit{\ with totally
geodesic boundary and }$Ric_{\tilde{g}}\geq 0$\textit{\ and }%
\begin{equation*}
\det \left( \mu (\tilde{g}^{-1}Ric_{\tilde{g}})\right) =f^{n}.
\end{equation*}%
\textit{Additionally, the set of all such solutions is compact in the }$%
C^{m} $\textit{\ -topology for any }$m\geq 0$\textit{.}
\end{corollary}

\begin{remark} The conformal problem with respect to the Ricci tensor has
been studied extensively. In \cite{LS} and \cite{GV3}, the authors studied
the negative Ricci curvature and proved that there exists a conformal metric
$\tilde{g}$ with negative Ricci tensor $Ric_{\tilde{g}}$ such that
\begin{equation*}
\det \left( \mu (\tilde{g}^{-1}Ric_{\tilde{g}})\right) =const..
\end{equation*}%
When the Ricci tensor is positive definite, in \cite{GW}, Guan and Wang
derived a conformal metric with a constant smallest eigenvalue of Ricci
tensor. In \cite{TW1}, Trudinger and Wang proved the prescribing problem of
positive Ricci tensor on closed manifold.
\end{remark}

This paper is organized as follows. We begin with some preliminaries in
Section 2. In Section 3, we will discuss the deformation and a priori
estimates. The proof of Theorem 1.1. is in Section 4. \vspace{3mm}

\section{PRELIMINARIES}

We introduce Fermi coordinates in a boundary neighborhood at first. In this
local coordinates, we take the geodesic in the inner normal direction $\nu =
\frac{\partial }{\partial x^{n}}$ parameterized by arc length, and $\left(
x^{1},...,x^{n-1}\right) $ forms a local chart on the boundary where $%
x^{n}=0 $. The metric can be expressed as
\begin{equation*}
g=g_{\alpha \beta }dx^{\alpha }dx^{\beta }+\left( dx^{n}\right) ^{2}.
\end{equation*}%
The Greek letters $\alpha ,\beta ,\gamma ,...$stand for the tangential
direction indices, $1\leq \alpha ,\beta ,\gamma ,...\leq n-1$, while the
Latin letters $i,j,k,...$ stand for the full indices, $1\leq i,j,k,...\leq n$
(See \cite{GrV} and \cite{Cn1}).

We denote the functions, tensors and covariant differentiations with respect
to the induced metric on the boundary by a $\mathit{bar}$(e.g. $\bar{\Gamma}%
_{\beta \gamma }^{\alpha }$, ${\bar{R}_{\alpha \beta }}$). Then the
Christoffel symbols on the boundary satisfy
\begin{equation*}
\bar{\Gamma}_{\alpha \beta }^{\gamma }=\frac{1}{2}g^{\gamma \delta }(\frac{%
\partial g_{\alpha \delta }}{\partial x^{\beta }}+\frac{\partial g_{\beta
\delta }}{\partial x^{\alpha }}-\frac{\partial g_{\alpha \beta }}{\partial
x^{\delta }})={\Gamma }_{\alpha \beta }^{\gamma },
\end{equation*}%
and ${\Gamma }_{nn}^{n}=0$, ${\Gamma }_{nn}^{\alpha }=0,$ ${\Gamma }%
_{n\alpha }^{n}=0$.

Let us denote $\frac{\partial }{\partial x^{i}}$ by $\partial _{i}$. The
boundary is called umbilic if the second fundamental form $L_{\alpha \beta
}=\tau g_{\alpha \beta }$, where $\tau $ is a function defined on $%
\partial M$. Since the boundary $\partial M$ is connected, by Schur Theorem,
$\tau =const.$. A totally geodesic boundary is umbilic with $\tau =0$.

Thus ${\Gamma }_{\alpha \beta }^{n}|_{\partial M}=L_{\alpha \beta }=\tau
g_{\alpha \beta }$ and ${\Gamma }_{n\beta }^{\alpha }|_{\partial
M}=-L_{\alpha \gamma }g^{\gamma \beta }=-\tau \delta _{\alpha }^{\beta }.\ \
$

Under the conformal metric $\tilde{g}=e^{-2u}g$, the functions, tensors and
the covariant differentiations with respect to $\tilde{g}$ denoted by a $%
\mathit{tilde}$ (e.g. $A_{\tilde{g}}$, ${\tilde{L}_{\alpha \beta }}$).

Let $\left[ g\right] $ be the set of metrics conformal to $g$. For $\tilde{g}%
=e^{-2u}g\in \left[ g\right] $, we consider the equation
\begin{equation}
F\left( \lambda (\tilde{g}^{-1}A_{\tilde{g}})\right) =f.  \tag{2.1}
\label{2.1}
\end{equation}
The Schouten tensor transforms according to the formula
\begin{equation*}
A_{\tilde{g}}=\nabla ^{2}u+du\otimes du-{\frac{1}{2}}|\nabla u|^{2}g+A_{g},
\end{equation*}%
where $\nabla u$ and $\nabla ^{2}u$ denote the gradient and Hessian of $u$
with respect to $g$. Consequently, ({\ref{2.1}}) is equivalent to
\begin{equation*}
F\Big(\lambda \Big(g^{-1}\Big[\nabla ^{2}u+du\otimes du-{\frac{1}{2}}|\nabla
u|^{2}g+A_{g}\Big]\Big)\Big)=f(x)e^{-2u}.
\end{equation*}

Then the second fundamental form satisfies
\begin{equation*}
\tilde{L}_{\alpha \beta }e^{u}=\frac{\partial u}{\partial \nu }g_{\alpha
\beta }+L_{\alpha \beta }.
\end{equation*}%
Note that the umbilicity is conformally invariant. When the boundary is
umbilic, the above formula becomes
\begin{equation*}
{\tilde{\tau}}e^{-u}=\frac{\partial u}{\partial \nu }+\tau ,
\end{equation*}%
where $\tilde{L}_{\alpha \beta }=\tilde{\tau}\tilde{g}_{\alpha \beta }$.

Therefore, whence the initial metric $g$ on manifold $M$ is with totally
geodesic boundary $\partial M$, the boundary of the manifold $M$ with
conformal metric $\tilde{g}=e^{-2u}g$ is still totally geodesic if and only
if $\frac{\partial u}{\partial \nu }=0$.

Therefore, in order to prove Theorem 1.1, we need to find admissible
solutions of the following equation

\begin{equation}
\left\{
\begin{array}{lr}
F\Big(\lambda \Big(g^{-1}\Big[\nabla ^{2}u+du\otimes du-{\frac{1}{2}}|\nabla
u|^{2}g+A_{g}\Big]\Big)\Big)=f(x)e^{-2u} & \text{in }M, \\
\frac{\partial u}{\partial \nu }=0\text{ \ \ \ \ \ \ \ \ \ \ \ \ \ \ \ \ \ \
\ \ \ \ \ \ \ \ \ \ \ \ \ \ \ \ \ \ \ \ \ \ \ \ \ \ \ \ \ \ \ \ \ \ \ \ \ }
& \text{on }\partial M.%
\end{array}%
\right.  \tag{2.2}  \label{2.2}
\end{equation}

\section{DEFORMATION, $C^1$  AND $C^2$ ESTIMATES }

To prove the existence of solution to the equation ({2.2}), we employ the
following deformation which defined in \cite{GV2}
\begin{equation}
\left\{
\begin{array}{lr}
F\Big(\lambda \Big(g^{-1}\Big[\varsigma (1-\psi (t))g+\psi (t)A_{g}+\nabla
^{2}u+du\otimes du-\frac{1}{2}|\nabla u|^{2}g\Big]\Big)\Big) &  \\
=\psi (t)f(x)e^{-2u}+(1-t)(\int e^{-(n+1)u})^{ {2}/{n+1}}\text{\ \ \  }  \text{
in }M, \\
\frac{\partial u}{\partial \nu }=0\text{ \ \ \ \ \ \ \ \ \ \ \ \ \ \ \ \ \ \
\ \ \ \ \ \ \ \ \ \ \ \ \ \ \ \ \ \ \ \ \ \ \ \ \ \ \ \  }
 \text{on }\partial M,%
\end{array}%
\right.  \tag{3.1}  \label{3.1}
\end{equation}%
where $\psi \in C^{1}[0,1]$ satisfies $0\leq \psi (t)\leq 1,\psi (0)=0$, $%
\psi (t)=1$ for $t\geq \frac{1}{2}$; and $\varsigma =(n\varrho
)^{-1}vol(M_{g})^{ \frac{2}{n+1}}$, where $F(1,\cdots ,1)=n\varrho $.

Similar as \cite{GV2}, at $t=1$, ({\ref{3.1}}) becomes ({\ref{2.2}}). While
at $t=0$, it becomes

\begin{equation*}
\left\{
\begin{array}{lr}
F\Big(\lambda\Big(g^{-1}\Big[\varsigma g+ \nabla ^{2}u + du\otimes du-\frac{1%
}{2} |\nabla u|^{2}g \Big]\Big)\Big)=(\int e^{-(n+1)u})^{\frac{2}{n+1}}\text{\  
}  \text{in }M, \\
\frac{\partial u}{\partial \nu }=0\text{ \ \ \ \ \ \ \ \ \ \ \ \ \ \ \ \ \ \
\ \ \ \ \ \ \ \ \ \ \ \ \ \ \ \ \ \ \ \ \ \ \ \ \ \ \ \ \ \ \ \ \ \ \ \ \ }
\text{on }\partial M.%
\end{array}%
\right.
\end{equation*}%
We can show that the above equation has a unique solution $u(x)\equiv 0$.

In fact, it is obvious that $u\equiv 0$ is a solution. Now we are going to
prove its uniqueness.

At the maximum point $x_{0}$ of $u$, no matter $x_{0}$ is interior or
boundary point, we always have that $\nabla u|_{x_{0}}=0,$ and $\nabla
^{2}u|_{x_{0}}$ is non-positive definite. In fact if $x_{0}$ is interior
point, it is clear; if $x_{0}$ is boundary point, we have $\frac{\partial u}{%
\partial \nu }|_{\partial M}=0$ by equation (\ref{3.1}), and $\frac{\partial
u}{\partial x^\alpha  }|_{x_0}=0$, where $\left\{ x^{\alpha }\right\}
_{1\leq \alpha \leq n-1}$ is a local coordinates on the boundary $\partial M$
around $x_{0}$. Therefore $\nabla ^{2}u|_{x_{0}}$ is non-positive definite.
Now at $x_{0}$ we have
\begin{eqnarray*}
vol(M_{g})^{\frac{2}{n+1}} &=&\varsigma \cdot n\varrho =\varsigma F(\lambda
(g^{-1}\cdot g)) \\
&\geq &F\Big(\lambda \Big(g^{-1}\Big[\varsigma g+\nabla ^{2}u+du\otimes du-%
\frac{1}{2}|\nabla u|^{2}g\Big]\Big)\Big) \\
&=&(\int e^{-(n+1)u})^{\frac{2}{n+1}}.
\end{eqnarray*}%
Similarly, at the minimum point of $u$, we can get $\varsigma \cdot n\varrho
\leq (\int e^{-(n+1)u})^{\frac{2}{n+1}}.$ As a result, we have $vol(M_{g})^{%
\frac{2}{n+1}}=\varsigma \cdot n\varrho =(\int e^{-(n+1)u})^{\frac{2}{n+1}}.$

By ($C_{6}$), we know $F\leq \varrho \sigma _{1}$. Hence,
\begin{eqnarray*}
\varsigma \cdot n\varrho &=&F\Big(\lambda \Big(g^{-1}\Big[\varsigma g+\nabla
^{2}u+du\otimes du-\frac{1}{2}|\nabla u|^{2}g\Big]\Big)\Big) \\
&\leq &\varrho \ \sigma _{1}\Big(\lambda \Big(g^{-1}\Big[\varsigma g+\nabla
^{2}u+du\otimes du-\frac{1}{2}|\nabla u|^{2}g\Big]\Big)\Big) \\
&=&\varrho \ \Big(n\varsigma +\triangle u+(1-\frac{n}{2})|\nabla u|^{2}\Big).
\end{eqnarray*}%
Then
\begin{equation*}
(\frac{n}{2}-1)\int_{M}|\nabla u|^{2}\leq \int_{M}\triangle u=\int_{\partial
M}\frac{\partial u}{\partial \nu }=0,
\end{equation*}%
and $u\equiv const.=0.$\newline

Thus the operator
\begin{eqnarray*}
\Psi _{t}[u] &=&F\Big(\lambda \Big(g^{-1}\Big[\varsigma (1-\psi (t))g+\psi
(t)A_{g}+\nabla ^{2}u+du\otimes du-\frac{1}{2}|\nabla u|^{2}g\Big]\Big)\Big)
\\
&&-\psi (t)f(x)e^{-2u}-(1-t)(\int e^{-(n+1)u})^{\frac{2}{n+1}}
\end{eqnarray*}%
satisfies Leray-Schauder degree $deg(\Psi _{0},\mathcal{O}_{0},0)\neq 0$ at $%
t=0$, where the Leray-Schauder degree is defined by \cite{Li1}(see \cite{Cn2}
for the boundary case) and $\mathcal{O}_{0}$ is a neighborhood of the zero
solution in $\{u\in C^{4,\alpha }(M):\frac{\partial u}{\partial \nu }=0\ $on
$\partial M\}$. Thus whence we obtain the homotopy-invariance of degree, we
can derive that the Leray-Schauder degree is nonzero at $t=1.$ This shows
that equation ({\ref{2.2}}) is solvable.\newline

The $C^{1}$ and $C^{2}$ estimates of the solutions to (\ref{3.1}) have been
proved in \cite{HS1}, we may obtain

\begin{lemma} \textit{For any fixed }$0<\delta <1$\textit{, there is a
constant }$C=C(\delta ,n,g,f)$\textit{\ such that any solution of (\ref{3.1}%
) with }$t\in \lbrack 0,1-\delta ]$\textit{\ satisfies }   $\|u\|_{C^{4,\alpha}} \leq C. $
\end{lemma}

So without loss of generality, we may assume that $u_{t_{i}}$ tends to $%
-\infty $ at $t_{i}\rightarrow 1$, where $u_{t_{i}}$ is the solution of ({%
\ref{3.1}}) at $t=t_{i}$ which will be denoted by $u_{i}$ in what follows.
Thus equation ({\ref{3.1}}) turns to be
\begin{equation}
\left\{
\begin{array}{lr}
\begin{array}{l}
F\Big(\lambda \Big(g^{-1}\Big[A_{g}+\nabla ^{2}u+du\otimes du-\frac{1}{2}%
|\nabla u|^{2}g\Big]\Big)\Big) \\
=(1-t)o+f(x)e^{-2u}\text{ }%
\end{array}
& \text{in }M, \\
\frac{\partial u}{\partial \nu }=0\text{ \ \ \ \ \ \ \ \ \ \ \ \ \ \ \ \ \ \
\ \ \ \ \ \ \ \ \ \ \ \ \ \ \ \ \ \ \ \ \ \ \ \ \ \ \ \ \ \ \ \ \ \ \ \ \ }
& \text{on }\partial M.%
\end{array}%
\right.  \tag{3.2}  \label{3.2}
\end{equation}%
where $u$ is assumed to be admissible, and $o\geq 0$ is a constant.\newline

Furthermore, we can get a more exact estimate on the geodesic ball $%
B(x,r)=\{y\in M\quad |\quad dist(x,y)<r\}$ :

\begin{lemma}(\cite{HS1}). \textit{Let }$u\in C^{4}(M)$\textit{\ be a }%
$k$\textit{\ -admissible solution of (\ref{3.1}) in }$B(x,r)$\textit{\ and }$%
0\leq r<1$\textit{\ . Then there is a constant }$C=C(n,g,f)$\textit{\ such
that }
\begin{equation}
\left( |\nabla ^{2}u|+|\nabla u|^{2}\right) (x^{\prime })\leq C\Big(%
r^{-2}+ \exp \left({-2 \inf_{B(x, 2 \sqrt{10} r)}u}\right)\Big).  \tag{3.3}  \label{3.3}
\end{equation}
\textit{for all }$x^{\prime }\in B(x,r)$\textit{.}
\end{lemma}

\section{PROOF OF THEOREM 1.1.}

We call $\left\{ u_{k}\right\} $ the blow up sequence and $\bar{x}\in M$ the
blow up point, if $u_{k}({{x}_{0,k}})\rightarrow -\infty $ as $%
x_{0,k}\rightarrow \bar{x},$ where $\left\{ x_{0,k}\right\} \subset M$. Now
let $\left\{ u_{k}\right\} $ be a blow up solutions of (\ref{3.2}) with the
blow up point $\bar{x}$.

First of all, we would like to prove that $\bar{x}$ can be approximated by
local minimum points of $u_k$. Let $v_k=e^{-(n-2)/2 u_k}$, denote ${v}%
_k(x_{0,k})^{\frac{1}{n-2}}$ by $R_{0,k}$ and $\frac{1}{ 1-e^{-1/2}}$ by $%
A_0 $.

\begin{lemma} \textit{In each geodesic ball }$%
B(x_{0,k},A_{0}R_{0,k}^{-1})\subset M$\textit{\ we may find a local maximum
point of }$v_{k}$\textit{, named by }$x_{k}$\textit{. Furthermore,}
\begin{equation*}
v_{k}(x_{k})=\sup_{B(x_{k},{v}_{k}(x_{k})^{-\frac{1}{n-2}})}v_{k}.
\end{equation*}
\end{lemma}

\begin{proof}

Let $e^{u_{k}(x_{0,k})}=\varepsilon _{0,k}.$ We define a
mapping:
\begin{eqnarray*}
\mathcal{U}_{0,k}:\mathcal{B}(0,{\varepsilon _{0,k}}^{-1/2})\subset
T_{x_{0,k}}(M) &\rightarrow &B(x_{0,k},{\varepsilon _{0,k}}^{1/2}) \\
y &\longmapsto &\exp _{x_{0,k}}(\varepsilon _{0,k}y),
\end{eqnarray*}%
where the metric on tangent space is ${\check{g}}_{k}=\varepsilon _{0,k}^{-2}%
\mathcal{U}_{0,k}^{\ast }{g}$ and $\mathcal{B}(0,{\varepsilon _{0,k}}%
^{-1/2}) $ is a geodesic ball. Moreover, consider a sequence of functions $%
\mu _{0,k}(y)=u_{k}(\mathcal{U}_{0,k}(y))-\log \varepsilon _{0,k}.$ We may
derive a equation that $\mu _{0,k}(y)$ satisfies. In fact, we have
\begin{equation*}
\left\{
\begin{array}{lr}
F\Big(\lambda \Big({\check{g}}_{k}^{-1}\Big[A_{{\check{g}}_{k}}+\nabla
^{2}\mu _{0,k}+d\mu _{0,k}\otimes d\mu _{0,k}-\frac{1}{2}|\nabla \mu
_{0,k}|^{2}{\check{g}}_{k}\Big]\Big)\Big) &  \\
=\varepsilon _{0,k}^{2}(1-t)o+f(\mathcal{U}_{0,k}(y))e^{-2\mu _{0,k}}
\ \ \ \ \ \ \ \ \ \ \ \ \text{in }\mathcal{B}(0,{\varepsilon _{0,k}}^{-1/2}), \\
\frac{\partial \mu _{0,k}}{\partial x^{n}}=0 \ \ \ \ \ \ \ \ \ \ \ \ \ \ \ \ \ \ \ \ \ \ \  \ \ \ \ \ \ \  \text{on }\mathcal{B}(0,{%
\varepsilon _{0,k}}^{-1/2})\cap \{x^{n}=0\}.%
\end{array}%
\right.
\end{equation*}%
where $\mu _{0,k}$ is admissible, and $o$ is a nonnegative constant.

Let us begin with the easy case $u_{k}\left( x\right) \geq u_{k}(x_{0,k})-1$
in $B(x_{0,k},{\varepsilon _{0,k}}^{1/2})$. In this case, $0\leq e^{-\frac{%
n-2}{2}\mu _{0,k}}\leq e^{\frac{n-2}{2}}$ in $\mathcal{B}(0,{\varepsilon
_{0,k}}^{-1/2})$. Hence, $\mu _{0,k}$ converges in $C^{3}$ to $\mu _{\infty
} $ with $0\leq e^{-\frac{n-2}{2}\mu _{\infty }}\leq e^{\frac{n-2}{2}}$ on $%
\mathbb{R}^{n}$. And the limit function $\mu _{\infty }$ satisfies
\begin{equation*}
F\Big(\lambda \Big(\delta ^{-1}\Big[\nabla ^{2}u+du\otimes du-\frac{1}{2}%
|\nabla u|^{2}\delta )\Big]\Big)\Big)=f(\bar{x})e^{-2u}.
\end{equation*}%
Then by the Liouville Theorem \cite{LL1}, we know that $0$ is the locally
minimum point of $\mu _{0,k}$. Rescaling back, we see that $x_{0,k}$ is the
locally minimum point of $u_{k}$ in $B(x_{0,k},{\varepsilon _{0,k}}^{1/2})$.

The alternative case is that there exists $x_{1,k}\in B(x_{0,k},{\varepsilon
_{0,k}}^{1/2})$ such that $u_{k}(x_{1,k})<u_{k}(x_{0,k})-1$. Then we may
consider the lower bound of $u_{k}$ in $B(x_{1,k},{\varepsilon _{1,k}}%
^{1/2}) $, where $\varepsilon _{1,k}=e^{u_{k}(x_{1,k})}<e^{-1}\varepsilon
_{0,k}$. If $u_{k}\geq u_{k}(x_{1,k})-1$ in $B(x_{1,k},{\varepsilon _{1,k}}%
^{1/2}),$ then $\mu _{1,k}(y)=u_{k}(\mathcal{U}_{1,k}(y))-\log \varepsilon
_{1,k}>-1$, where
\begin{equation*}
\mathcal{U}_{1,k}:y\rightarrow \exp _{x_{1,k}}(\varepsilon _{1,k}y),
\end{equation*}%
and $x_{1,k}$ is a locally minimum point of $u_{k}$.

Otherwise, we may repeat the previous proceedings with $%
u_{k}(x_{j,k})<u_{k}(x_{j-1,k})-1$ ($x_{j,k}\in B(x_{j-1,k},{\varepsilon
_{j-1,k}}^{1/2})$), $\varepsilon _{j,k}=e^{u_{k}(x_{j,k})}<e^{-1}\varepsilon
_{j-1,k}$ and $\mu _{j,k}(y)=u_{k}(\mathcal{U}_{j,k}(y))-\log \varepsilon
_{j,k}$, where
\begin{equation*}
\mathcal{U}_{j,k}:y\rightarrow \exp _{x_{j,k}}(\varepsilon _{j,k}y).
\end{equation*}%
For any given $k$, as $u_{k}\in C^{\infty }\left( M\right) $, there exists $%
j(k)\in \mathbb{N}$, $j\left( k\right) <\infty $ such that $u_{k}(x_{j\left(
k\right) ,k})<u_{k}(x_{j\left( k\right) -1,k})-1$ and $u_{k}\geq
u_{k}(x_{j\left( k\right) ,k})-1$ in $B(x_{j\left( k\right) ,k},{\varepsilon
_{j\left( k\right) ,k}}^{1/2})$. Hence, we can find a locally minimum point
of the $u_{k}$ in $B(x_{j\left( k\right) ,k},{\varepsilon _{j\left( k\right)
,k}}^{1/2})\subset B(x_{0,k},A_{0}{\varepsilon _{0,k}}^{1/2}).$ This
completes the proof (See Lemma 3.2 in \cite{TW1} for more details).
\end{proof}

Now we consider the rescaled sequence $w_{k}=u_{k}-\sup_{M}u_{k}$. Suppose $%
x_{k}^{0}$ is the maximum point of $u_{k}$. Since $e^{-2\sup
u_{k}}f(x_{k}^{0})=e^{-2u_{k}(x_{k})}f(x_{k}^{0})\leq C(\triangle
u_{k}+A_{g})(x_{k}^{0})\leq C$ thus $\bar{x}=\lim x_{k}$ is the blow up
point with respect to $w_{k}$ as well. It is obviously that $w_{k}$
satisfies the equation
\begin{equation*}
\left\{
\begin{array}{lr}
F\Big(\lambda \Big(g^{-1}\Big[A_{g}+\nabla ^{2}w_{k}+dw_{k}\otimes dw_{k}-%
\frac{1}{2}|\nabla w_{k}|^{2}g\Big]\Big)\Big) &  \\
=(1-t)o+f(x)e^{-2\sup_{M}u_{k}}e^{-2w_{k}}\text{ } & \text{in }M, \\
\frac{\partial w_{k}}{\partial \nu }=0\text{ \ \ \ \ \ \ \ \ \ \ \ \ \ \ \ \
\ \ \ \ \ \ \ \ \ \ \ \ \ \ \ \ \ \ \ \ \ \ \ \ \ \ \ \ \ \ \ \ \ \ \ \ \ \
\ } & \text{on }\partial M.%
\end{array}%
\right.
\end{equation*}%
where $w_{k}$ is admissible, and $o\geq 0$ is a constant.

By virtue of Lemme 4.1, we may assume $\bar{x}=\lim x_{k}$, where $\left\{
x_{k}\right\} $ are locally minimum points of $u_{k}$. Hence $\left\{
x_{k}\right\} $ are also locally minimum points of $w_{k}$ and
\begin{equation*}
w_{k}(x_{k})=\inf_{B(x_{k},e^{\frac{1}{2}{w}_{k}(x_{k})})}w_{k}.
\end{equation*}%
Note that $F$ satisfies $(C_{1})-(C_{6})$ and $w_{k}$ are $\Gamma $
admissible, where $\Gamma \subset \Sigma _{\frac{1}{n-2}}$. Hence $w_{k}$
are subharmonic and satisfy
\begin{equation}
W+\frac{1}{n-2}\sigma _{1}(W)g\geq 0,  \tag{4.1}  \label{4.1}
\end{equation}%
where $W=\nabla ^{2}w_{k}+dw_{k}\otimes dw_{k}-\frac{1}{2}|\nabla
w_{k}|^{2}g+A_{g}$. We need the idea of the minimal radial functions of $w$
in $B_{R}(x_{0})$ (\cite{TW1}):
\begin{equation*}
\widehat{w}(x)=\sup \{{w}(y):y\in \partial B_{r}(x_{0}),r=d(x,x_{0})\leq R\},
\end{equation*}%
and denote $\nabla ^{2}\widehat{w}+d\widehat{w}\otimes d\widehat{w}-\frac{1}{%
2}|\nabla \widehat{w}|^{2}g+A_{g}$ by $\widehat{W}$. Now we are ready to
prove the following

\begin{proposition} \textit{Let }$u_{j}$\textit{\ be a blow up
sequence of solutions to (\ref{3.2}). Then }$w_{j}=u_{j}-\sup_{M}u_{j}$%
\textit{\ converges in }$w^{1,p}$\textit{\ (for any }$1<p<\frac{n}{n-1}$%
\textit{\ ) to an admissible function }$w$\textit{. Moreover, if }$\bar{x}$%
\textit{\ is a blow up point of }$w$\textit{, then near }$\bar{x}$\textit{, }%
\begin{equation}
w(x)=2\log d(x,\bar{x})+o(1),  \tag{4.2}  \label{4.2}
\end{equation}%
\textit{where }$d(x,\bar{x})$\textit{\ denotes the geodesic distance from }$x
$\textit{\ to }$\bar{x}$\textit{\ with respect to the metric }$g$\textit{.
Furthermore, each blow up point is isolated.}
\end{proposition}

\begin{proof}Since a similar proposition on manifold without boundary
has appeared in \cite{TW1}, we only focus on the differences.

\textit{Step 1.} We may get admissible solutions on the doubled manifold.
Glue two copies of $(M,g)$ along the totally geodesic boundary together and
denote the doubling manifold by $\check{M}$. With the given smooth
Riemannian metric $g$ on $M$, there is a standard metric $\check{g}$ on $%
\check{M}$ induced from $g$. When $\partial M$ is totally geodesic in $(M,g)$%
, $\check{g}$ is $C^{2,1}$ on $\check{M}$ (see \cite{E}).

We can extend $w_{k}$ to a $C^{2}(\check{M})$ function $\check{w}_{k}$ as
follows: Near the boundary we take Fermi coordinates, $\check{w}_{k}$ is
then defined as
\begin{equation*}
\check{w}_{k}(x_{1},\cdots ,x_{n})=\left\{
\begin{array}{ll}
w_{k}(x_{1},\cdots ,x_{n}), & x_{n}\geq 0, \\
w_{k}(x_{1},\cdots ,-x_{n}), & x_{n}\leq 0.%
\end{array}%
\right.
\end{equation*}

Since $\frac{\partial w_{k}}{\partial \nu }=0$, it is easy to verify by
definition that $\check{w}_{k}\in C^{2}(\check{M})$. As matter of fact,
\begin{eqnarray*}
&&\lim_{x_{n}\rightarrow 0^{+}}\frac{\partial \check{w}_{k}}{\partial x^{n}}%
(x_{1},\cdots ,x_{n})=\frac{\partial w_{k}}{\partial x^{n}}(x_{1},\cdots
,x_{n-1},0) \\
&=&0=-\frac{\partial w_{k}}{\partial x^{n}}(x_{1},\cdots
,x_{n-1},0)=\lim_{x_{n}\rightarrow 0^{-}}\frac{\partial \check{w}_{k}}{%
\partial x^{n}}(x_{1},\cdots ,x_{n}),
\end{eqnarray*}%
and
\begin{equation*}
\lim_{x_{n}\rightarrow 0^{+}}\frac{\partial ^{2}\check{w}_{k}}{\partial
(x^{n})^{2}}(x_{1},\cdots ,x_{n})=\lim_{x_{n}\rightarrow 0^{-}}\frac{%
\partial ^{2}\check{w}_{k}}{\partial (x^{n})^{2}}(x_{1},\cdots ,x_{n}).
\end{equation*}%
Thus from the admissible property of ${w}_{k}$ we know that $\check{w}_{k}$
is also admissible and satisfies (4.1).

\textit{Step 2.} We can find convergent "minimal radial functions" on
doubled manifold. Inequality ({4.1}) says $\check{w}_{k}$ is subharmonic.
From Corollary 2.1 in \cite{TW1}, $\left\{ \check{w}_{k}\right\} $ converges
to a subharmonic function $\check{w}$ in $W^{1,p}$ (for any $1<p<\frac{n}{n-1%
})$). By Corollary 2.2 in \cite{TW1}, the corresponding minimal radial
functions $\widehat{\check{w}}_{k}$ also converge to $\widehat{\check{w}}$.
Note that the minimal radial functions depend only on distance to the
center, by Corollary 2.1 and Corollary 2.2 in \cite{TW1}, we may obtain
\begin{equation}
\widehat{\check{w}}\left( r\right) =\lim_{k\rightarrow \infty }\widehat{%
\check{w}}_{k}\left( r\right) ,  \tag{4.3}  \label{4.3}
\end{equation}%
where
\begin{equation*}
\widehat{\check{w}}_{k}(r)=\sup \{\check{w}_{k}(y):y\in \partial
B_{r}(x_{k})\},
\end{equation*}%
and
\begin{equation*}
\widehat{\check{w}}(r)=\sup \{\check{w}(y):y\in \partial B_{r}(\bar{x})\}.
\end{equation*}%
On the one hand, based on (4.3) and (4.1), we can get the following
estimates
\begin{equation}
\widehat{\check{w}}(x)\leq 2\log d(x,\bar{x})+C.  \tag{4.4}  \label{4.4}
\end{equation}%
In fact, we may assume $\widehat{\check{w}}_{k}(r)=\check{w}_{k}(x_{r})$ , $%
x_{r}=(0,\cdots ,0,r)$, $|A_{g}|\leq Cr/2$.  $\widehat{\check{w}}_{k}$ are
still admissible and satisfy inequality (\ref{4.1}). Thus
\begin{eqnarray*}
0 &\leq &\left( (n-2)\widehat{W}_{nn}+\displaystyle\Sigma _{i}\widehat{W}%
_{ii}\right) (x_{r})\allowdisplaybreaks \\
&\leq &(n-1)\left( \widehat{\check{w}}_{k}^{\prime \prime }+(\widehat{\check{%
w}}_{k}^{\prime })^{2}-\frac{g_{nn}}{2}(\widehat{\check{w}}_{k}^{\prime
})^{2}+Cr/2\right) \allowdisplaybreaks \\
&&+\displaystyle\Sigma _{i=1}^{n-1}\left( (\frac{1}{r}+C)\widehat{\check{w}}%
_{k}^{\prime }-\frac{g_{ii}}{2}(\widehat{\check{w}}_{k}^{\prime
})^{2}+Cr/2\right) \allowdisplaybreaks \\
&\leq &(n-1)\left( \widehat{\check{w}}_{k}^{\prime \prime }+\frac{1}{r}%
\widehat{\check{w}}_{k}^{\prime }+C(\widehat{\check{w}}_{k}^{\prime
}+r)\right) ,
\end{eqnarray*}%
where the last inequality comes from $\Sigma _{i}g_{ii}\geq n$. Hence,
\begin{equation*}
\left( \log (r\widehat{\check{w}}_{k}^{\prime }+r^{2})\right) ^{\prime
}+C\geq 0.
\end{equation*}%
By taking a limit we get (\ref{4.4}).

On the other hand, let $\widehat{\check{v}}_{k}=e^{-(n-2)/2\widehat{%
\check{w}}_{k}}$. From $\triangle \widehat{\check{v}}_{k}\leq C\widehat{%
\check{v}}_{k}r$, we get
\begin{equation*}
\lbrack r^{n-1}\widehat{\check{v}}_{k}^{\prime }]^{\prime }\leq Cr^{n}%
\widehat{\check{v}}_{k}
\end{equation*}%
Thus, by a direct calculation we know
\begin{equation*}
\widehat{\check{w}}(x)\geq 2\log d(x,\bar{x})+o(1).
\end{equation*}%
Therefore
\begin{equation}
\widehat{\check{w}}(x)=2\log d(x,\bar{x})+o(1).  \tag{4.5}  \label{4.5}
\end{equation}

Then the comparison principle helps us to deduce (4.2) from (4.5). Roughly
speaking, since $\check{w} $ equals $\widehat{\check{w}} $ at some points,
the comparison principle implies they are equal everywhere. That is
\begin{equation*}
\check{w}(x)=2\log d(x,\bar{x})+o(1).
\end{equation*}%
(For more details, one may consult section 3 of \cite{TW1}.) \end{proof}

\textit{\ Proof of Theorem 1.1.}

As the proof of Proposition 4.2, we glue two copies of $(M,g)$ together.
Denote the doubled manifold and functions by a "check" (e.g. $\check{M},%
\check{w}$). Since the Ricci curvature $Ric_{e^{-2\check{w}}g}$ is still
non-negative, by (\ref{4.5}) and Volume Comparison Theorem, there is at most
one end away from the blow up points; the metric $e^{-2\check{w}}g$ is in
fact a Euclidean one (see section 7 of \cite{GV1} for details), namely $(M,g)$ is conformally equivalent to the unit half sphere, which
contradicts with the assumption in Theorem 1.1. Therefore there is
a unform $L^{\infty }$ bound for solutions. So the set of solutions is
compact. This completes the proof of Theorem 1.1. $\ \ \ \ \ \ \ \ \ \ \
\square $

\begin{Acknowledgement}The first author would like to thank her advisor,
Professor Kefeng Liu, for his support and encouragement.
\end{Acknowledgement}

\end{document}